\documentclass[oneside,english]{amsart}
\usepackage[T1]{fontenc}
\usepackage[latin9]{inputenc}
\usepackage{amsthm}
\usepackage{amssymb}

\makeatletter
\numberwithin{equation}{section}
\numberwithin{figure}{section}
\theoremstyle{plain}
\newtheorem{thm}{\protect\theoremname}
\theoremstyle{definition}
\newtheorem{defn}[thm]{\protect\definitionname}

\makeatother

\usepackage{babel}
\providecommand{\definitionname}{Definition}
\providecommand{\theoremname}{Theorem}

\begin{document}
\title[{\small{}on preserving notion of largeness}]{{\large{}a combinatorial viewpoint on preserving notion of largeness
and an abstract rado theorem}}
\author{\noindent Sayan Goswami }
\address{{\large{}Department of Mathematics, University of Kalyani, Kalyani-741235,
Nadia, West Bengal, India.}}
\email{{\large{}sayan92m@gmail.com}}
\author{Subhajit Jana}
\address{{\large{}Department of Mathematics, University of Kalyani, Kalyani-741235,
Nadia, West Bengal, India.}}
\email{{\large{}suja12345@gmail.com}}
\begin{abstract}
{\large{}In this paper we will systematically study the preservation
of the notion of largeness of sets, arises from the algebraic structure
of Stone-\v{C}ech compactification, under homomorphism and difference
group. Some of these results were studied previously using algebra
of Stone-\v{C}ech compactification. But our approach is purely combinatorial.}{\large\par}
\end{abstract}

\maketitle

\section{{\large{}introduction}}

{\large{}Studying large sets which contain special Ramsey theoretic
configurations has a long history. There are several approaches including
Topological dynamics, ergodic theory, algebra of Stone-\v{C}ech compactification,
combinatorics etc. to study these large sets. In some recent work
\cite{key-1,key-3,key-5,key-6} it has been studied that how these
large sets behave under homomorphism and difference group. It has
been shown that under some restrictions, these large sets are well
preserved under homomorphism whether they are preserved in difference
group with no extra condition. But those proofs use the algebra of
Stone-\v{C}ech compactification whereas we provide here combinatorial
proofs of those results with some new results. Also we will prove
an abstract version of Rado's theorem.}{\large\par}

{\large{}To state the results, we need some definitions and properties
of different large sets which is described below. For details the
readers are invited to read \cite{key-4}.}{\large\par}

{\large{}For any set $X$, let $\mathcal{P}_{f}\left(X\right)$ be
the set of finite nonempty subsets of $X$. Given a subset $A$ of
a semigroup $\left(S,\cdot\right)$, and $x\in S$, let $x^{-1}A=\left\{ y\in S:x\cdot y\in A\right\} $.}{\large\par}
\begin{defn}
{\large{}\cite{key-4} Let $\left(S,\cdot\right)$ be a semigroup
and $A\subseteq S$, then}{\large\par}
\end{defn}

\begin{enumerate}
\item {\large{}The set $A$ is $\mathit{piecewise\:syndetic}$ if and only
if there exists $G\in\mathcal{P}_{f}\left(S\right)$ such that for
every $F\in\mathcal{P}_{f}\left(S\right)$, there exists $x\in S$
such that $F\cdot x\subseteq\bigcup_{t\in G}t^{-1}A$.}{\large\par}
\item {\large{}$\mathcal{T}=^{\mathbb{N}}S$}{\large\par}
\item {\large{}For $m\in\mathbb{N}$, $\mathcal{J}_{m}=$$\left\{ \begin{array}{cc}
\left(t\left(1\right),t\left(2\right),\ldots,t\left(m\right)\right)\in\mathbb{N}^{m} & :\\
t\left(1\right)<t\left(2\right)<\ldots<t\left(m\right)
\end{array}\right\} $}{\large\par}
\item {\large{}Given $m\in\mathbb{N}$, $a\in S^{m+1}$, $t\in\mathcal{J}_{m}$
and $f\in\mathcal{T}$, 
\[
x\left(m,a,t,f\right)=\left(\prod_{j=1}^{m}\left(a\left(j\right)\cdot f\left(t\left(j\right)\right)\right)\right)\cdot a\left(m+1\right)
\]
}{\large\par}
\item {\large{}$A\subseteq S$ is called a $J-set$ iff for each $F\in\mathcal{P}_{f}\left(\mathcal{T}\right)$,
there exists $m\in\mathbb{N}$, $a\in S^{m+1}$, $t\in\mathcal{J}_{m}$
such that, for each $f\in\mathcal{T}$,}\\
{\large{}
\[
x\left(m,a,t,f\right)\in A.
\]
}{\large\par}
\end{enumerate}
{\large{}The following combinatorial definition is quite complicated
but we have to use this.}{\large\par}
\begin{defn}
{\large{}\label{coll} \cite{key-4} Let $\left(S,\cdot\right)$ be
a semigroup and $\mathcal{A}\subseteq\mathcal{P}\left(S\right)$.
Then $\mathcal{A}$ is collection-wise piecewise syndetic if and only
if there exists functions $K:\mathcal{P}_{f}\left(A\right)\rightarrow\mathcal{P}_{f}\left(S\right)$
and $x:\mathcal{P}_{f}\left(A\right)\times\mathcal{P}_{f}\left(S\right)\rightarrow S$
such that for all $F\in\mathcal{P}_{f}\left(S\right)$ and all $\mathcal{F}$
and $\mathcal{H}$ in $\mathcal{P}_{f}\left(A\right)$ with $\mathcal{F}\subseteq\mathcal{H}$
one has $F\cdot x\left(\mathcal{H},F\right)\subseteq\bigcup_{t\in K\left(\mathcal{F}\right)}t^{-1}\left(\bigcap\mathcal{F}\right)$.}{\large\par}
\end{defn}

{\large{}When we write that $\langle C_{F}\rangle_{F\in\mathcal{I}}$
is a downward directed family, we mean that $\left(\mathcal{I},\geq\right)$
is a directed set and when $F,G\in\mathcal{I}$with $F\geq G$, one
has $C_{F}\subseteq C_{G}$.}{\large\par}
\begin{thm}
{\large{}\label{thm } \cite{key-4} Let $\left(S,\cdot\right)$ be
a semigroup and $A\subseteq S$}{\large\par}

{\large{}(a) The set $A$ is said to be Quasi-central if and only
if there is a downward directed family $\langle C_{F}\rangle_{F\in\mathcal{I}}$
of subsets of $A$ such that,}{\large\par}
\begin{enumerate}
\item {\large{}\label{2.1} for each $F\in\mathcal{I}$ and each $x\in C_{F}$,
there exists $G\in\mathcal{I}$ such that $C_{G}\subseteq x^{-1}C_{F}$
and}{\large\par}
\item {\large{}\label{2.2} for each $F\in\mathcal{I}$,$C_{F}$ is piecewise
syndetic.}{\large\par}
\end{enumerate}
{\large{}(b) The set $A$ is said to be central if and only if there
is a downward directed family $\langle C_{F}\rangle_{F\in\mathcal{I}}$
of subsets of $A$ such that,}{\large\par}
\begin{enumerate}
\item {\large{}\label{2.3} for each $F\in\mathcal{I}$ and each $x\in C_{F}$,
there exists $G\in\mathcal{I}$ such that $C_{G}\subseteq x^{-1}C_{F}$
and}{\large\par}
\item {\large{}\label{2.4} $\left\{ C_{F}:F\in\mathcal{I}\right\} $ is
collectionwise piecewise syndetic.}{\large\par}
\end{enumerate}
{\large{}(c) The set $A$ is said to be $C$-set if and only if there
is a downward directed family $\langle C_{F}\rangle_{F\in\mathcal{I}}$
of subsets of $A$ such that,}{\large\par}
\begin{enumerate}
\item {\large{}\label{2.5} for each $F\in\mathcal{I}$ and each $x\in C_{F}$,
there exists $G\in\mathcal{I}$ such that $C_{G}\subseteq x^{-1}C_{F}$
and}{\large\par}
\item {\large{}\label{2.6} for each $F\in\mathcal{I}$,$\bigcap_{F\in\mathcal{F}}C_{F}$
is a $J$-set.}{\large\par}
\end{enumerate}
\end{thm}

{\large{}Now we systematically study that, many large sets in a commutative
semigroup $\left(S,+\right)$ is also large in it's difference group
$\left(S-S,+\right)$. To prove these results we will use the combinatorial
characterizations of these sets instead of using the structure of
Stone-\v{C}ech compactification of $\left(S,+\right)$ and $\left(S-S,+\right)$.
So we will not recall the structure of Stone-\v{C}ech compactification
of $\left(S,+\right)$.}{\large\par}

{\large{}For a commutative cancellative semigroup $\left(S,+\right)$,
the difference group $\left(S-S,+\right)$ consists of elements of
the form $\left\{ a-b:a,b\in S\right\} $ where $a-b$ is defined
to be that element for which $b+\left(a-b\right)=a$. It is easy to
check that $\left(S-S,+\right)$ is a commutative group with identity
$0$, where $a+0=a$. As each element $a\in S$ can be embedded in
$S-S$ as $a\rightarrow\left(a+b\right)-b$ for some fixed $b\in S$,
we see that $S\subseteq S-S$.}{\large\par}

{\large{}In \cite{key-1,key-3} it was proved that, if $\Phi:\left(S,\cdot\right)\rightarrow\left(T,\cdot\right)$
is a semigroup homomorphism and $\Phi\left(S\right)$ is large in
some sense in $T$, then for any large set $A$, $\Phi\left(A\right)$
is large in $T$. In this paper we will prove those results using
combinatorial tools and we will prove a new result that if $\Phi\left(S\right)$
is a $J$-set in $T$, then for any $J$-set set $A$, $\Phi\left(A\right)$
is a $J$-set in $T$.}{\large\par}

{\large{}The paper is arranged as, section 2 is devoted to the study
of large sets in difference group and section 3 is devoted to the
study of large sets under homomorphism and section 4 is devoted to
the abstract version of Rado's theorem regarding inhomogeneous partition
regularity.}{\large\par}

\section{{\large{}largeness in difference group}}
\begin{thm}
{\large{}\label{thick} Let $\left(S,+\right)$ be a commutative cancellative
semigroup and $A\subseteq S$ is thick in $S$, then $A$ is thick
in $S-S$.}{\large\par}
\end{thm}

\begin{proof}
{\large{}Let us take any finite subset of $S-S$, say $F\in\mathcal{P}_{f}(S-S)$.
Let $F=\{x_{1}-y_{1},x_{2}-y_{2},\ldots,x_{n}-y_{n}\}$ where $x_{i},y_{i}\in S$
for $i=1,2,\ldots,\,n$.}\\
{\large{}Consider, 
\begin{align*}
G & =F+(y_{1}+y_{2}+\ldots+y_{n})\\
 & =\left\{ x_{1}+\sum_{i=2}^{n}y_{i},\ldots x_{j}+\sum_{\underset{i\neq j}{i=1}}^{n}y_{i},\ldots,x_{n}+\sum_{i=1}^{n-1}y_{i}\right\} 
\end{align*}
Then $G\in\mathcal{P}_{f}(S)$. Now as $A\subset S$ is thick in $S$,
there exists $y\in S$ such that $G+y\subset A$. This implies that
$F+(y_{1}+y_{2}+\ldots+y_{n})+y\subset A$, i.e. $F+(y+\sum_{i=1}^{n}y_{i})\subset A$.
So, $A$ is thick in $S-S$.}{\large\par}
\end{proof}
{\large{}Note that syndetic set in $\left(S,+\right)$ may not be
syndetic in $\left(S-S,+\right)$ as, $2\mathbb{N}$ is syndetic in
$\left(\mathbb{N},+\right)$ but not in $\left(\mathbb{Z},+\right)$.
But the notion of piecewise syndetic sets are preserved in difference
group. it can be deduced from \cite[Lemma 1.12]{key-5} as:}\\
{\large{}Let $A\subset S$ is piecewise syndetic in $S$ then there
exists $x\in S$ such that $-x+A$ is central in $S$. But then from
\cite[Lemma 1.12]{key-5} , it is also central in $S-S$, hence piecewise
syndetic in $S-S$. As translation of piecewise syndetic sets are
still piecewise syndetic, we have $\left(-x+A\right)+x\subseteq A$
is piecewise syndetic in $\left(S-S,+\right)$. }{\large\par}

{\large{}Now here we give a purely combinatorial approach:}{\large\par}
\begin{thm}
{\large{}\label{piecewise} Let $\left(S,+\right)$ be a commutative
cancellative semigroup and $A\subseteq S$ is piecewise syndetic in
$S$, then $A$ is piecewise syndetic in $S-S$.}{\large\par}
\end{thm}

\begin{proof}
{\large{}Let $A\subseteq S$ is piecewise syndetic in $S$. So there
exists a set $F\in\mathcal{P}_{f}\left(S\right)$ such that $\bigcup_{t\in F}-t+A$
is thick in $S$. So by theorem \ref{thick}, $\bigcup_{t\in F}-t+A$
is thick in $S-S$and so $A$ is piecewise syndetic in $S-S$.}{\large\par}
\end{proof}
\begin{thm}
{\large{}\label{6} Let $\left(S,+\right)$ be a commutative cancellative
semigroup and$A\subseteq S$ is quasi central in $S$, then $A$ is
quasi central in $S-S$.}{\large\par}
\end{thm}

\begin{proof}
{\large{}As $A\subseteq S$ is quasi-central there exists a decreasing
sequence$\langle C_{F}\rangle_{F\in\mathcal{I}}$ of subsets of $A$
guaranteed by theorem \ref{thm }. From the theorem \ref{piecewise}
$C_{F}$ is piecewise syndetic in $S-S$ for all $F\in\mathcal{I}$
as they are piecewise syndetic in $S$. Hence the conclusion follows
immediately.}{\large\par}
\end{proof}
\begin{thm}
{\large{}Let $\left(S,+\right)$ be a commutative cancellative semigroup
and $A\subseteq S$ is central in $S$, then $A$ is central in $S-S$.}{\large\par}
\end{thm}

\begin{proof}
{\large{}As we have already proved the quasi central set is preserved,
we only need to prove that if $\mathcal{A}$ is collectionwise piecewise
syndetic in $S$ then it is in $S-S$. Let the functions $K$ and
$x$ are guaranteed by definition \ref{coll}. For any $F\in\mathcal{P}_{f}\left(S-S\right)$
fix an element $\theta_{F}\in S$ such that $F+\theta_{F}\in S$.
Now define a fuction $\theta:\mathcal{P}_{f}\left(S-S\right)\rightarrow S$
by $\theta\left(F\right)=\theta_{F}$. Now define $y:\mathcal{P}_{f}\left(A\right)\times\mathcal{P}_{f}\left(S-S\right)\rightarrow S-S$
by $y\left(\mathcal{H},F\right)=x\left(\mathcal{H},\left(F+\theta\left(F\right)\right)\right)+\theta\left(F\right)$.
Then for any $F\in\mathcal{P}_{f}\left(S-S\right)$, $F+y\left(\mathcal{H},F\right)=F+x\left(\mathcal{H},\left(F+\theta\left(F\right)\right)\right)+\theta\left(F\right)\subseteq\bigcup_{t\in K\left(\mathcal{F}\right)}t^{-1}\left(\bigcap\mathcal{F}\right)$.
So, $\mathcal{A}$ is collectionwise piecewise syndetic in $S-S$
and this proves the result.}{\large\par}
\end{proof}
{\large{}For algebraic proof of the following theorem one can see
\cite[Lemma 3.11]{key-6}.}{\large\par}
\begin{thm}
{\large{}Let $\left(S,+\right)$ be a commutative cancellative semigroup
and $A\subseteq S$ is  $J$-set in $S$, then $A$ is  $J$-set in
$S-S$.}{\large\par}
\end{thm}

\begin{proof}
{\large{}Let $A\subseteq S$ is $J$-set in $S$. Let us choose a
finite set $F\in\mathcal{P}_{f}((S-S)^{\mathbb{N}})$ arbitrarily
and let $F=\{f_{1},f_{2},\ldots,f_{k}\}$.}\\
{\large{}For each $i=1,2,\ldots,k$, let $f_{i}(t)=a_{i}\left(t\right)-b_{i}\left(t\right)$,
where $a_{i}\left(t\right),b_{i}\left(t\right)\in S$ for all $t\in\mathbb{N}$.
Let us define a function $g\in S^{\mathbb{N}}$ by,
\[
g(t)=\sum_{i=1}^{k}b_{i}\left(t\right)
\]
Now for each $i=1,2,\ldots,k$, take the functions $h_{i}\in S^{\mathbb{N}}$
by $h_{i}\left(t\right)=f_{i}\left(t\right)+g\left(t\right)$. }{\large\par}

{\large{}Take the new collection $H\in\mathcal{P}_{f}(S^{\mathbb{N}})$
as $H=\left\{ h_{i}\right\} _{i=1}^{k}$. As $A\subseteq S$ is $J$-set
in $S$, there exists $a\in S$ and $K\in\mathcal{P}_{f}\left(\mathbb{N}\right)$
such that 
\[
a+\sum_{t\in K}h_{i}(t)\in A
\]
for all $i=1,2,\ldots,k$. Now 
\[
a+\sum_{t\in K}h_{i}(t)=a+\sum_{t\in K}\left(f_{i}\left(t\right)+g\left(t\right)\right)\in A
\]
Let $b=\sum_{t\in K}\sum_{i=1}^{k}b_{i}\left(t\right)$ then $a+b+\sum_{t\in K}f_{i}\left(t\right)\in A$.
Now $a+b\in S\subset S-S$ and so $A$ is a $J$-set in $S-S$.}{\large\par}
\end{proof}
\begin{thm}
{\large{}Let $\left(S,+\right)$ be a commutative cancellative semigroup
and $A\subseteq S$ is  $C$-set in $S$, then $A$ is  $C$-set in
$S-S$.}{\large\par}
\end{thm}

\begin{proof}
{\large{}The proof is similar to the proof of theorem \ref{6} and
so left to the reader.}{\large\par}
\end{proof}

\section{{\large{}largeness under homomorphism}}
\begin{thm}
{\large{}If $\Phi:(S,\cdot)\rightarrow(T,\cdot)$ be a semigroup homomorphism
and $\Phi(S)$ is piecewise syndetic in $T$ then $A\subseteq S$
is piecewise syndetic implies $\Phi(A)$ is also piecewise syndetic
in $T$.}{\large\par}
\end{thm}

\begin{proof}
{\large{}Let $A\subseteq S$ is piecewise syndetic. So there exists
$H\in\mathcal{P}_{f}(S)$ such that for all $F\in\mathcal{P}_{f}(S)$
there exists $x\in S$ , $F\cdot x\subseteq\bigcup_{t\in H}t^{-1}A$.
So, $\bigcup_{t\in H}\Phi(t)^{-1}\Phi(A)$ is thick in $\Phi(S)$
as, $\Phi(\bigcup_{t\in H}t^{-1}A)\subseteq\bigcup_{t\in H}\Phi(t)^{-1}\Phi(A)$.}{\large\par}

{\large{}As $\Phi(S)$ is piecewise syndetic in $T$, there exists
$K\in\mathcal{P}_{f}(T)$ such that for all $G\in\mathcal{P}_{f}(T)$
there exists $y\in S$ , $G\cdot y\subseteq\bigcup_{s\in K}s^{-1}\Phi(S)$.}{\large\par}

{\large{}Choose $F\in\mathcal{P}_{f}(T)$ and $z\in T$ such that
$F\cdot z\subset\bigcup_{s\in K}s^{-1}\Phi(S)$.}{\large\par}

{\large{}Let $F=\{a_{1},a_{2},\dots,a_{n}\}$, so we have $\{a_{1}z,a_{2}z,\dots,a_{n}z\}\subseteq\bigcup_{s\in K}s^{-1}\Phi(S)$
and this implies $\{s_{1}a_{1}z,s_{2}a_{2}z,\dots,s_{n}a_{n}z\}\subseteq\Phi(S)$
where $s_{1},s_{2},\dots,s_{n}\in K$, not necessarily being distinct.}{\large\par}

{\large{}So there exists $y$ such that $\{s_{1}a_{1}z,s_{2}a_{2}z,\dots,s_{n}a_{n}z\}\cdot y\subset\bigcup_{t\in H}\Phi(t)^{-1}\Phi(A)$
which implies $F\cdot zy\subseteq\bigcup_{s\in K}\bigcup_{t\in H}s^{-1}\Phi(t)^{-1}\Phi(A)=\bigcup_{t\in H,s\in K}\left(\Phi\left(t\right)s\right)^{-1}\Phi(A)$.
Hence $\Phi(A)$ is piecewise syndetic in $T$.}{\large\par}
\end{proof}
\begin{thm}
{\large{}\label{11}If $\Phi:(S,\cdot)\rightarrow(T,\cdot)$ be a
semigroup homomorphism and $\Phi(S)$ is piecewise syndetic in $T$
then $A\subseteq S$ is quasi central implies $\Phi(A)$ is also quasi
central in $T$.}{\large\par}
\end{thm}

\begin{proof}
{\large{}It follows from the above theorem and the chain condition
\ref{thm }. So we leave it to the reader.}{\large\par}
\end{proof}
\begin{thm}
{\large{}If $\Phi:(S,\cdot)\rightarrow(T,\cdot)$ be a semigroup homomorphism
and $\Phi(S)$ is piecewise syndetic in $T$ then $A\subseteq S$
is central implies $\Phi(A)$ is also central in $T$.}{\large\par}
\end{thm}

\begin{proof}
{\large{}Let $\Phi:(S,\cdot)\rightarrow(T,\cdot)$ be a semigroup
homomorphism. As we have proved the quasi central set is preserved
under $\Phi$, we have to prove only that, if $\mathcal{A}$ is collectionwise
picewise syndetic in $S$, then $\Phi(\mathcal{A})=\{\Phi(A):A\in\mathcal{A}\}$
is collectionwise picewise syndetic in $T$. Let $K$ and $x$ are
functions arised from the definition of central set.}{\large\par}

{\large{}As $\Phi(S)$ is piecewise syndetic in $T$, there exists
a finite subset $G$ of $T$ such that $\bigcup_{t\in G}t^{-1}\Phi(S)$
is thik in $T$. Let $F\in\mathcal{P}_{f}(T)$ be arbitrary and so
, there exists $y\in T$ such that, $F\cdot y\subseteq\bigcup_{t\in G}t^{-1}\Phi(S)$.}{\large\par}

{\large{}So for each $F\in\mathcal{P}_{f}(T)$ fix an element $y\in T$
and $\Psi_{F}\in\mathcal{P}_{f}(\Phi(S))$ such that $F\cdot y\subseteq\bigcup_{t\in G}t^{-1}\Psi_{F}$.}{\large\par}

{\large{}Now define,}{\large\par}

{\large{}1. $\eta:\mathcal{P}_{f}(T)\rightarrow\Phi(S)$ by $\eta\left(F\right)=\Psi_{F}$}{\large\par}

{\large{}2. $\kappa:\mathcal{P}_{f}(T)\rightarrow T$ by $\kappa\left(F\right)=y$}{\large\par}

{\large{}3. For each $B\in\mathcal{P}_{f}\left(\Phi\left(S\right)\right)$
fix $B^{\prime}\in\mathcal{P}_{f}(S)$ such that $\Phi(B^{\prime})=B$,
where $\Phi\left(B^{\prime}\right)=\{\Phi\left(x\right):x\in B\}$}{\large\par}

{\large{}4. For each $\mathcal{B}\in\mathcal{P}_{f}\left(\Phi\left(\mathcal{A}\right)\right)$
fix $\theta_{\mathcal{B}}\in\mathcal{P}_{f}\left(\mathcal{A}\right)$
such that $\Phi\left(\theta_{\mathcal{B}}\right)=\mathcal{B}$}{\large\par}

{\large{}5. $K_{1}:\mathcal{P}_{f}\left(\Phi\left(\mathcal{A}\right)\right)\rightarrow\mathcal{P}_{f}\left(T\right)$
by $K_{1}\left(\mathcal{\mathcal{B}}\right)=\{\Phi\left(t\right)\cdot s:t\in K\left(\theta_{\mathcal{B}}\right),s\in G\}$}{\large\par}

{\large{}6. $z:\mathcal{P}_{f}\left(\Phi\left(\mathcal{A}\right)\right)\times\mathcal{P}_{f}\left(T\right)\rightarrow T$
by, $z\left(\mathcal{H},F\right)=\kappa\left(F\right)\cdot\Phi\left(x\left(\theta_{\mathcal{H}},\Psi_{F}^{\prime}\right)\right)$}\linebreak{}

{\large{}Note that for $F\in\mathcal{P}_{f}\left(T\right)$, $\mathcal{F}\subseteq\mathcal{H}$
in $\mathcal{P}_{f}\left(\mathcal{A}\right)$, so $\theta_{\mathcal{F}}\subseteq\theta_{\mathcal{H}}$
in $\mathcal{A}$.}{\large\par}

{\large{}Now from the definition as $\mathcal{A}$ is collectionwise
picewise syndetic in $S$, we have}{\large\par}

{\large{}$\Psi_{F}^{\prime}\cdot x\left(\theta_{\mathcal{H}},\Psi_{F}^{\prime}\right)\subseteq\bigcup_{t\in K\left(\theta_{\mathcal{F}}\right)}t^{-1}\left(\bigcap\theta_{\mathcal{F}}\right)$}{\large\par}

{\large{}$i.e,\Phi\left(\Psi_{F}^{\prime}\right)\cdot\Phi\left(x\left(\theta_{\mathcal{H}},\Psi_{F}^{\prime}\right)\right)\subseteq\bigcup_{t\in K\left(\theta_{\mathcal{F}}\right)}\Phi\left(t\right){}^{-1}\Phi\left(\bigcap\theta_{\mathcal{F}}\right)$}{\large\par}

{\large{}$i.e,\Phi\left(\Psi_{F}^{\prime}\right)\cdot\Phi\left(x\left(\theta_{\mathcal{H}},\Psi_{F}^{\prime}\right)\right)\subseteq\bigcup_{t\in K\left(\theta_{\mathcal{F}}\right)}\Phi\left(t\right){}^{-1}\left(\bigcap\mathcal{F}\right)$}{\large\par}

{\large{}$i.e,\Psi_{F}\cdot\Phi\left(x\left(\theta_{\mathcal{H}},\Psi_{F}^{\prime}\right)\right)\subseteq\bigcup_{t\in K\left(\theta_{\mathcal{F}}\right)}\Phi\left(t\right){}^{-1}\left(\bigcap\mathcal{F}\right)$}{\large\par}

{\large{}But $F\cdot y\subseteq\bigcup_{s\in G}s^{-1}\Psi_{F}$.}{\large\par}

{\large{}So we have,
\begin{align*}
F\cdot y\cdot\Phi\left(x\left(\theta_{\mathcal{H}},\Psi_{F}^{\prime}\right)\right) & \subseteq\bigcup_{s\in G}s^{-1}\Psi_{F}\cdot\Phi\left(x\left(\theta_{\mathcal{H}},\Psi_{F}^{\prime}\right)\right)\\
 & \subseteq\bigcup_{s\in G,t\in K\left(\theta_{\mathcal{F}}\right)}s^{-1}\Phi\left(t\right){}^{-1}\left(\bigcap\mathcal{F}\right)\\
 & =\bigcup_{s\in G,t\in K\left(\theta_{\mathcal{F}}\right)}\left(\Phi\left(t\right)\cdot s\right){}^{-1}\left(\bigcap\mathcal{F}\right)\\
 & \subseteq\bigcup_{l\in K_{1}\left(\mathcal{F}\right)}l^{-1}\left(\bigcap\mathcal{F}\right)
\end{align*}
}{\large\par}

{\large{}Which implies $F\cdot z\left(\mathcal{H},F\right)\subseteq\bigcup_{l\in K_{1}\left(\mathcal{F}\right)}l^{-1}\left(\bigcap\mathcal{F}\right)$.
Hence as we have defined above, this proves that $\Phi(\mathcal{A})$
is collectionwise picewise syndetic in $T$. }{\large\par}

{\large{}Therefore $\Phi(A)$ is central in $T$.}{\large\par}
\end{proof}
\begin{thm}
{\large{}\label{13} Let $\Phi:(S,\cdot)\rightarrow(T,\cdot)$ be
a semigroup homomorphism and $\Phi(S)$ is $J$-set in $T$. If $A\subseteq S$
is $J$-set then $\Phi(A)$ is also $J$-set in $T$.}{\large\par}
\end{thm}

\begin{proof}
{\large{}First note that as, $A\subseteq S$ is $J$-set $\Phi(A)$
is a $J$-set in $\Phi(S)$.}{\large\par}

{\large{}Let $F\in\mathcal{P}_{f}\left(^{\mathbb{N}}T\right)$ and
$m\in\mathbb{N}$, $a\in T^{m+1}$, $t\in J_{m}$ i.e. $t(1)<t(2)<\dots<t(m)$
in $\mathbb{N}$ such that for each $f\in F$, $x\left(m,a,t,f\right)\in\Phi(S)$.
Let $t'=\max\left\{ t(1),t(2),\dots,t(m)\right\} $. For each $f\in F$,
let $g_{f}\in{}^{\mathbb{N}}T$ such that $g_{f}\left(n\right)=f\left(t'+n\right),\forall n\in\mathbb{N}.$
Define $G\in\mathcal{P}_{f}\left(^{\mathbb{N}}T\right)$ by $G=\left\{ g_{f}:f\in F\right\} $.
Then again as $\Phi(S)$ is $J$-set in $T$ and $G\in\mathcal{P}_{f}\left(^{\mathbb{N}}T\right)$
we can apply the above argument on $G$.}{\large\par}

{\large{}Using the above argument repeatedly we obtain $\left\{ m_{n}\right\} {}_{n=1}^{\infty}$,
$\left\{ a_{n}\right\} {}_{n=1}^{\infty}$, and $\left\{ t_{n}\right\} {}_{n=1}^{\infty}$
such that $x\left(m_{n},a_{n},t_{n},f\right)\in\Phi\left(S\right)$
for all $f\in F$ and $\max\left\{ t_{n}\right\} <\min\left\{ t_{n+1}\right\} $.}{\large\par}

{\large{}Let $F^{\prime}\in\mathcal{P}_{f}\left(^{\mathbb{N}}\Phi\left(S\right)\right)$
is defined by $F^{\prime}=\left\{ \left\{ x\left(m_{n},a_{n},t_{n},f\right)\right\} {}_{n=1}^{\infty}:f\in F\right\} $.}{\large\par}

{\large{}Now as, $\Phi\left(A\right)$ is a $J$-set in $\Phi\left(S\right)$,
there exists $m\in\mathbb{N}$, $a\in\Phi\left(S\right){}^{m+1}$,
$t\in J_{m}$ i.e. $t(1)<t(2)<\dots<t(m)$ in $\mathbb{N}$ such that
for each $g\in F^{\prime}$ , $x\left(m,a,t,g\right)\in\Phi(A)$.}{\large\par}

{\large{}Now 
\begin{align*}
x(m,a,t,g) & =\left(\prod_{i=1}^{m}a(i)g(t(i))\right)\cdot a(m+1)\\
 & =\left(\prod_{i=1}^{m}a(i).x(m_{t(i)},a_{t(i)},t_{t(i)},f)\}_{n=1}^{\infty}\right)\cdot a(m+1)\\
 & =\left(\prod_{i=1}^{n}b(i).f_{t(i)}\right)\cdot b(n+1)
\end{align*}
}{\large\par}

{\large{}for some $n\in\mathbb{N}$, $b=\left(b\left(1\right),b\left(2\right),\ldots,b\left(n\right)\right)\in T^{n+1}$,
$t\in J_{n}$ with $b(n+1)=a(m+1)$. The last line was found by expanding
the previous one. Hence $\Phi\left(A\right)$ is a $J$-set in $T$.}{\large\par}
\end{proof}
\begin{thm}
{\large{}Let $\Phi:(S,\cdot)\rightarrow(T,\cdot)$ be a semigroup
homomorphism and $\Phi(S)$ is $J$-set in $T$. If $A\subseteq S$
is $C$-set then $\Phi(A)$ is also $C$-set in $T$.}{\large\par}
\end{thm}

\begin{proof}
{\large{}It is similar as \ref{11} and so the proof is left to the
reader.}{\large\par}
\end{proof}

\section{{\large{}abstract Rado theorem}}

{\large{}In \cite{key-8} Rado proved that the system of equation
$Ax=b$ is partition regular over $\mathbb{Z}$i.e., each entries
of $x$ are in same color for any finite coloring of $\mathbb{Z}$
if and only if it has a constant solution. Then the above result was
further generalized for various cases. Recently in \cite{key-7} it
has been established that the result is still true if one consider
the entries of the matrix from a ring and consider the equation over
any ring $R$. Their result is,}{\large\par}
\begin{thm}
{\large{}Let $A$ be an $m\times n$ matrix with entries in a ring
$R$, and let $b\in R^{m}$ be non-zero. Then the system of equations
$Ax=b$ is partition regular over $R$ if and only if it has a constant
solution.}{\large\par}
\end{thm}

{\large{}Now to generalize the above theorem for groups we need an
analogous notion of matrix for groups. In \cite[Definition 7.3]{key-2}
it has been shown that one can proceed by considering homomorphism.
The proof of the following version is similar to the proof of \cite[Theorem 2]{key-7}.}{\large\par}
\begin{thm}
{\large{}Let $G$ be a commutative group with identity $0$. Let $k,n\in\mathbb{N}$
and let $A:G^{n}\rightarrow G^{k}$be a homomorphism and $b\in G^{k}$
be non zero. Then the $Ax=b$, where $x\in G^{n}$ is partition regular
over $G$, if and only if it has a constant solution.}{\large\par}
\end{thm}

\begin{proof}
{\large{}Let $G$ be a commutative group with identity $0$. Let $k,n\in\mathbb{N}$
and let $A:G^{n}\rightarrow G^{k}$be a homomorphism. For each $i=1,2,\dots,n$
let $c_{i}:G\rightarrow G^{k}$ be the map defined by $c_{i}(y)=A(0,0,\dots,y,0,\dots,0)$,
where $y$ appears in the $i$-th position. }{\large\par}

{\large{}Let $H=\left\{ \sum_{i=1}^{n}c_{i}(y):y\in G\right\} $.
One can easily check that $H$ is a subgroup of $G^{k}$. }{\large\par}

{\large{}Suppose that there is no constant solution (i.e. entries
of $x$ are not constant) : this means that $b$ does not belongs
to the subgroup $H$ of $G^{k}$.}{\large\par}

{\large{}Let us $K$ be the maximal subgroup of $G^{k}$ which contains
$H$ but not the element $b$. Then the quotient map $\theta$ from
$G^{k}$ to $G^{k}/K$ has $\theta(b)$ non-zero and also every non
trivial subgroup of $G^{k}/K$ must contain $\theta(b)$ (by the maximality
of $K$). It follows from this that $G^{k}/K$ is either a cyclic
group of prime power order or else the group $\mathbb{Z}_{p^{\infty}}$
of all $p^{k}$-th root of unity for any $k$ , for some fixed prime
$p$. In each case this is isomorphic to a subgroup of the circle
$\mathbb{T}$. Hence there is a group homomorphism $\phi$ from $G^{k}$
to the circle $\mathbb{T}$ such that $\phi(H)=0$ and $\phi(b)\neq0$.}{\large\par}

{\large{}Now define a coloring of $G$ with $d^{n}$coloring, where
$d$ is large enough positive integer, by coloring $t\in G$ with
the $n$ tupple}\\
{\large{} $\left(f\left(\phi\left(c_{1}\left(t\right)\right)\right),f\left(\phi\left(c_{2}\left(t\right)\right)\right),\dots,f\left(\phi\left(c_{n}\left(t\right)\right)\right)\right)$,
where $f$ is the map that sends interval of the circle with arguments
in $\left[2\pi j/d,2\pi(j+1)/d\right]$ to $j$, for each $0\leq j\leq d-1$.
Suppose that for this coloring we have a monochromatic vector $x$
in $G^{n}$such that $Ax=b$. We have that
\begin{align*}
\phi\left(\sum_{i=1}^{n}c_{i}\left(x_{i}\right)\right) & =\phi\left(\sum_{i=1}^{n}A(0,0,\dots,x_{i},0,\dots,0)\right)\\
 & =\phi\left(Ax\right)\\
 & =\phi(b)
\end{align*}
 where $x=\left(x_{1},\dots,x_{n}\right)$. But $\phi\left(\sum_{i=1}^{n}c_{i}\left(x_{1}\right)\right)=0$
as $\sum_{i=1}^{n}c_{i}\left(x_{1}\right)\in H$.}{\large\par}

{\large{}So we have,}\\
{\large{}
\[
\sum_{i=1}^{n}\left(\phi\left(c_{i}\left(x_{i}\right)\right)-\phi\left(c_{i}\left(x_{1}\right)\right)\right)=\phi\left(\sum_{i=1}^{n}c_{i}\left(x_{i}\right)\right)-\phi\left(\sum_{i=1}^{n}c_{i}\left(x_{1}\right)\right)=\phi(b).
\]
 But this is a contradiction for $d$ large, as each term in the sum
on the left hand side has argument in $\left[0,4\pi/d\right]$, by
the definition of coloring.}{\large\par}
\end{proof}
\textbf{\large{}Acknowledgment: }{\large{}The first author acknowledges
the grant UGC-NET SRF fellowship with id no. 421333 of CSIR-UGC NET
2016.}{\large\par}


\begin{thebibliography}{1}
{\large{}\bibitem{key-1} V. Bergelson and D. Glasscock, On the interplay
between additive and multiplicative largeness and its combinatorial
applications, J. Combin. Theory Ser. A, 172(2020)105203}{\large\par}

{\large{}\bibitem{key-2}V. Bergelson, J. H. Johnson Jr. and J. Moreira,
New polynomial and multidimensional extensions of classical partition
results, Journal of Combinatorial Theory, Series A 147 (2017) 119--154.}{\large\par}

{\large{}\bibitem{key-3} N. Hindman and D. Strauss, A simple characterization
of sets satisfying the Central Sets Theorem, New York J. Math. 15
(2009), 405-413.}{\large\par}

{\large{}\bibitem{key-4} N. Hindman and D. Strauss, Algebra in the
Stone-Cech Compactication: The-ory and Applications, 2nd edition,
Walter de Gruyter \& Co., Berlin, 2012.}{\large\par}

{\large{}\bibitem{key-5}N. Hindman and D. Strauss, Image partition
regularity of matrices over commutative semigroups, Topology and its
Applications 259 (2019), 179-202.}{\large\par}

{\large{}\bibitem{key-6} N. Hindman and D. Strauss,Image partition
regular matrices and concepts of largeness, New York J. Math. 26 (2020),
230-260.}{\large\par}

{\large{}\bibitem{key-7} I. Leader and P. Russell, Inhomogeneous
Partition Regularity, Electronic J. combinatorics, Volume 27, Issue
2 (2020)}{\large\par}

{\large{}\bibitem{key-8} R. Rado, Studien zur Kombinatorik, Math.
Zeit. 36 (1933), 242-280.}{\large\par}
\end{thebibliography}
\end{document}